\documentclass[12pt]{amsproc}

\usepackage{amsmath,amssymb}

\copyrightinfo{2000}{American Mathematical Society}

\newcommand{\cD}{{\mathcal D}}

\newcommand{\bN}{{\mathbb N}}

\newcommand{\bR}{{\mathbb R}}
\newcommand{\bS}{{\mathbb S}}
\newcommand{\bC}{{\mathbb C}}
\newcommand{\bZ}{{\mathbb Z}}

\newtheorem{definition}{definition}[section]

\newtheorem{theorem}{Theorem}[section]
\newtheorem{proposition}{Proposition}[section]

\numberwithin{equation}{section}

\begin{document}

\title{Distributions associated to homogeneous distributions}

\author{A.~V.~Kosyak}
\address{Institute of Mathematics, Ukrainian National Academy of Sciences,
3 Tere\-shchenkivs'ka, Kyiv, 01601, Ukraine.}
\email{kosyak01@yahoo.com, kosyak@imath.kiev.ua}

\author{V.~I.~Polischook}
\address{Dept. of Mathematics, St. Petersburg State Polytechnical University,
Polytechnicheskaya 29, 195251, St. Petersburg, Russia}
\email{polischook@list.ru}

\author{V.~M.~Shelkovich}
\address{Department of Mathematics, St.-Petersburg State Architecture
and Civil Engineering University, 2 Krasnoarmeiskaya 4, 190005,
St. Petersburg, Russia.}
\email{shelkv@vs1567.spb.edu}

\thanks{The authors would like to thank the Institute of Applied Mathematics,
University of Bonn for the hospitality. The first and the third authors were
supported in part by DFG Project AL 214/41-1.}


\date{}


\begin{abstract}
In this paper we continue to study {\it quasi associated homogeneous
distributions \rm{(}generalized functions\rm{)}} which were introduced
in the paper by V.M. Shelkovich, Associated and quasi associated homogeneous
distributions (generalized functions), J. Math. An. Appl., {\bf 338}, (2008), 48-70.
For the multidimensional case we give the characterization of these
distributions in the terms of the dilatation operator $U_{a}$ (defined as
$U_{a}f(x)=f(ax)$, $x\in \bR^n$, $a >0$) and its generator
$\sum_{j=1}^{n}x_j\frac{\partial }{\partial x_j}$.
It is proved that $f_k\in {\cD}'(\bR^n)$ is a quasi associated homogeneous
distribution of degree~$\lambda$ and of order $k$ if and only if
$\bigl(\sum_{j=1}^{n}x_j\frac{\partial}{\partial x_j}-\lambda\bigr)^{k+1}f_{k}(x)=0$,
or if and only if $\bigl(U_a-a^\lambda I\bigr)^{k+1}f_k(x)=0$, $\forall \, a>0$,
where $I$ is a unit operator. The structure of a quasi associated homogeneous
distribution is described.
\end{abstract}

\maketitle

\setcounter{equation}{0}

\section{Introduction}
\label{s1}

\subsection{Homogeneous distributions.\/}\label{s1.1}
Let us consider one-parameter multiplicative group $G$ of the {\em dilation operators}
$U_{a}$, $a\in\bR_{+}=\{a\in\bR, a>0\}$, which acts in $C^\infty(\bR^n\setminus\{0\})$
as follows:
$$
C^\infty(\bR^n\setminus\{0\})\ni \phi(x) \to U_a\phi(x)=\phi(ax)=\phi(ax_1,\dots,ax_n), \quad x\ne 0,
\quad a\in\bR_+.
$$

Non zero function $\phi\in C^\infty(\bR^n\setminus\{0\})$ is \emph{eigenfunction}
of a group $G$ if $U_{a}\phi=\mu_a\phi$ for all $a>0$.
If $\mu_a\equiv 1$ for all $a>0$, such an invariant function $\phi$ is
called {\it homogeneous}.

\begin{proposition}
\label{hom-functions}
A function $\phi$ is eigenfunction of a group $G$ if and only if
it can be represented in the form $\phi=r^\lambda h$, where
$r(x)=\|x\|$, $\lambda\in\bC$ and $h\ne 0$ is a homogeneous function.
Moreover, to any eigen-element $\phi$ corresponds a unique $\lambda\in \bC$
such that $\mu(a)=a^\lambda$ for all $a>0$, where $\bC$ is the set of complex numbers.
\end{proposition}

\begin{proof}
Sufficiency of the representation $\phi(x)=r^\lambda h(x)$ is evident.
Suppose that $0\ne\phi\in C^\infty(\bR^n\setminus\{0\})$
and $\phi(ax)=\mu_a\phi(x)$ for all $x\ne 0$ and $a>0$.
Let $\phi(x_*)\ne 0$. Since $\phi$ is a continuous function then $f(a)=\phi(ax_*)$
is continuous in $\bR_+$. As far as $f(a)=\mu_a\phi(x_*)$, we conclude that a function
$\mu_a$ is continuous. Next, we have
$$
\mu_{ab}\phi(x_*)=\phi(abx_*)=\mu_{a}\phi(bx_*)=\mu_{a}\mu_{b}\phi(x_*)\,.
$$
Since $\phi(x_*)\ne 0$, therefore $\mu_{ab}=\mu_{a}\mu_{b}$.
All continuous non-zero solutions of this equation have the form $\mu_{a}=a^\lambda$,
where $\lambda \in \bC$. Thus, $\phi(ax)=a^\lambda\phi(x)$ for all $a>0$ and for any $x\ne 0$,
where $\lambda$ is fixed for any function $\phi$. Setting $a=1/r(x)$, we obtain the representation
$$
\phi(x)=r^\lambda(x)\phi\Big(\frac{x}{r(x)}\Big)=r^\lambda(x)h(x)\,,
$$
where $h(x)=\phi(x/r(x))\ne 0$ is a {\em homogeneous function}.
\end{proof}

According to Proposition~\ref{hom-functions}, a \emph{spectral subspace} of a group $G$
(of order zero) corresponding to the number $\lambda$ is the following
$$
S_0(\lambda)=\bigcap_{a>0}\ker(U_a-a^\lambda I),
$$
where $I$ is a unit operator.

We shall now carry some of the preceding results to the case of distributions.
\begin{definition}
\label{de2} \rm
(~\cite[Ch.I,\S 3.11.,(1)]{G-Sh},~\cite[3.2.]{Hor})
A distribution $f_0\in {\cD}'(\bR^n)$ is said to be {\it homogeneous {\rm(}HD{\rm)}}
of degree~$\lambda$ if for any $a >0$ and $\varphi \in {\cD}(\bR^n)$ we have
$$
\Bigl\langle f_0(x),\varphi\Big(\frac{x}{a}\Big) \Bigr\rangle
=a^{\lambda+n} \bigl\langle f_0(x),\varphi(x) \bigr\rangle,
$$
i.e.,
\begin{equation}
\label{1*}
U_{a}f_0(x)=f_0(ax)=a^{\lambda}f_0(x), \quad x\in \bR^n.
\end{equation}
\end{definition}

The Euler differential operator
\begin{equation}
\label{17.1***-00}
\sum_{j=1}^{n}x_j\frac{\partial }{\partial x_j}
\end{equation}
is a generator of one-parameter group of dilatation $G=\{U_{a}:a>0\}$.

Recall a well-known Euler theorem.
\begin{theorem}
\label{th5}
{\rm(}see~{\rm\cite[Ch.III,\S 3.1.]{G-Sh}}{\rm)}
A distribution $f_{0}\in {\cD}'(\bR^n)$ is homogeneous of degree~$\lambda$ if and only if
it satisfies the Euler equation
$$
\sum_{j=1}^{n}x_j\frac{\partial f_{0}}{\partial x_j}=\lambda f_{0}.
$$
\end{theorem}

\subsection{Associated homogeneous distributions.\/}\label{s1.2}
A concept of an {\it associated homogeneous distribution
{\rm(}AHD{\rm)}} (for the one-dimensional case) was first introduced
and studied in the book~\cite[Ch.I,\S 4.1.]{G-Sh}. Later, {\it AHDs}
were studied in~\cite{E-K1},~\cite{E-K2}.
AHDs were introduced by the following way.

It is naturally to study {\it associated functions} of various
orders~\cite[Ch.I,\S 4.1.]{G-Sh} of an arbitrary linear transformation $U$.

Functions $f_1,f_2,\dots,f_k,\dots$
are said to be {\it associated with the eigenfunction} $f_0$ of the
transformation $U$  if
\begin{equation}
\label{1}
\begin{array}{rcl}
\displaystyle
Uf_0&=&cf_0, \\
\displaystyle
Uf_k&=&cf_k+df_{k-1}, \quad k=1,2,\dots, \\
\end{array}
\end{equation}
where $c$, $d$ are constants. Consequently, $U$ reproduces an
{\it associated function} of $k$th order up to some {\it associated
function} of $(k-1)$th order.

Since according to Definition~\ref{de2}, HD of degree~$\lambda$
is an eigenfunction of {\it any} dilatation operator $U_{a}$,
$a>0$, with the eigenvalue $a^{\lambda}$, in the
book~\cite[Ch.I,\S 4.1.]{G-Sh}, by analogy with Definition~(\ref{1}),
the following definition is introduced: a function $f_1(x)$ is said
to be {\it associated homogeneous} of order $1$ and of degree~$\lambda$
if for any $a >0$
\begin{equation}
\label{1.1}
f_1(ax)=a^{\lambda}f_1(x)+h(a)f_{0}(x),
\end{equation}
where $f_{0}$ is a homogeneous function of degree $\lambda$. Here,
in view of (\ref{1*}) and (\ref{1}), $c=a^{\lambda}$. Next,
in~\cite[Ch.I,\S 4.1.]{G-Sh} it is proved that up to a constant factor
$h(a)=a^{\lambda}\log{a}$.
Thus, by setting in the relation (\ref{1.1}) $c=a^{\lambda}$ and
$d=h(a)=a^{\lambda}\log{a}$, Definition~(\ref{1.1}) takes the following form:

\begin{definition}
\label{de3} \rm
(~\cite[Ch.I,\S4.1.,(1),(2)]{G-Sh})
A distribution $f_1\in {\cD}'(\bR)$ is called {\it associated
homogeneous {\rm(}AHD{\rm)}} of order $1$ and of
degree~$\lambda$ if for any $a >0$ and $\varphi \in {\cD}(\bR)$
$$
\Bigl\langle f_1,\varphi\Big(\frac{x}{a}\Big) \Bigr\rangle
=a^{\lambda+1} \bigl\langle f_1,\varphi \bigr\rangle
+ a^{\lambda+1}\log{a}\bigl\langle f_{0},\varphi \bigr\rangle,
$$
i.e.,
$$
U_{a}f_1(x)=f_1(ax)=a^{\lambda}f_1(x)+a^{\lambda}\log{a}f_{0}(x),
$$
where $f_{0}$ is a homogeneous distribution of degree $\lambda$.
\end{definition}

In~\cite{Sh}, it was proved that there exist {\it only\/} AHDs of order
$k=0$, i.e., HDs (given by Definition~\ref{de2}) and of order $k=1$
(given by Definition~(\ref{1.1}) or Definition~\ref{de3}).

\subsection{Quasi associated homogeneous distribution.}\label{s1.3}
In view of the above facts, in~\cite{Sh}, a definition of {\it quasi associated
homogeneous distribution {\rm(}QAHD{\rm)}} was introduced, which is a natural
generalization of the notion of an {\it associated eigenvector} (\ref{1}).

\begin{definition}
\label{de6} \rm
A distribution $f_k\in {\cD}'(\bR)$ is said to be {\it quasi associated
homogeneous\/} of degree~$\lambda$ and of order $k$, $k=0,1,2,3,\dots$
if for any $a >0$ and $\varphi \in {\cD}(\bR)$
$$
\Bigl\langle f_k(x),\varphi\Big(\frac{x}{a}\Big) \Bigr\rangle
=a^{\lambda+1} \langle f_k(x),\varphi(x) \rangle
+\sum_{r=1}^{k}h_r(a)\langle f_{k-r}(x),\varphi(x) \rangle,
$$
i.e.,
\begin{equation}
\label{50}
U_{a}f_k(x)=f_k(ax)=a^{\lambda}f_k(x)+\sum_{r=1}^{k}h_r(a) f_{k-r}(x),
\quad \quad x\in \bR,
\end{equation}
where $f_{k-r}(x)$ is a QAHD of degree $\lambda$ and of order
$k-r$, \ $h_r(a)$ is a differentiable function, $r=1,2,\dots,k$.
Here for $k=0$ we suppose that sums in the right-hand sides of the
above relations are empty.
\end{definition}

According to~\cite[Theorem~3.2.]{Sh}, in order to introduce
a QAHD of degree $\lambda$ and order $k$ instead of Definition~\ref{de6}
one can use the following definition.

\begin{definition}
\label{de7} \rm
A distribution $f_{k}\in {\cD}'(\bR)$ is called a {\it QAHD} of
degree~$\lambda$ and of order $k$, $k=0,1,2,\dots$,
if for any $a >0$ and $\varphi \in {\cD}(\bR)$
\begin{equation}
\label{17**}
\Bigl\langle f_{k}(x),\varphi\Big(\frac{x}{a}\Big) \Bigr\rangle
=a^{\lambda+1} \bigl\langle f_{k}(x),\varphi(x) \bigr\rangle
+\sum_{r=1}^{k}
a^{\lambda+1}\log^r{a}\bigl\langle f_{k-r}(x),\varphi(x) \bigr\rangle,
\end{equation}
i.e.,
\begin{equation}
\label{17.1**}
U_{a}f_k(x)=f_k(ax)=a^{\lambda}f_{k}(x)+\sum_{r=1}^{k}a^{\lambda}\log^r{a} f_{k-r}(x),
\quad \quad x\in \bR,
\end{equation}
where $f_{k-r}(x)$ is an QAHD of degree $\lambda$ and of order $k-r$,
\ $r=1,2,\dots,k$. Here for $k=0$ we suppose that the
sums in the right-hand sides of (\ref{17**}), (\ref{17.1**}) are empty.
\end{definition}

In~\cite{Sh} the following multidimensional analog of Definition~\ref{de7} was introduced.
\begin{definition}
\label{de9} \rm
We say that a distribution $f_{k} \in {\cD}'(\bR^n)$ is a {\it QAHD\/}
of degree $\lambda$ and of order~$k$, \ $k=0,1,2,\dots$, if for any
$a >0$ we have
\begin{equation}
\label{17.1***}
U_{a}f_k(x)=f_{k}(ax)
=a^{\lambda}f_{k}(x)+\sum_{r=1}^{k}a^{\lambda}\log^r{a} f_{k-r}(x),
\quad x\in \bR^n,
\end{equation}
where $f_{k-r}(x)$ is a QAHD of degree $\lambda$ and
of order $k-r$, \ $r=1,2,\dots,k$. (Here we suppose that for $k=0$ the
sum in the right-hand side of (\ref{17.1***}) is empty.)
\end{definition}

\section{Quasi associated homogeneous distributions and nilpotency}
\label{s2}

\subsection{Structure of the spectral subspace.}\label{s2.1}
Let us define the {\em spectral subspace} of order $k\in \bZ_+=\{0,1,2\dots\}$
corresponding to $\lambda$:
$$
S_k(\lambda)=\bigcap_{a>0}\ker(U_a-a^\lambda I)^{k+1}, \quad k\in \bZ_+.
$$
It is clear that
$$
S_0(\lambda)\subseteq S_1(\lambda)\subseteq\dots
\subseteq S_k(\lambda)\subseteq\dots
$$

Below we describe all functions which belong to the spectral space $S_k(\lambda)$.

\begin{theorem}
\label{reduction}
Let $f\in {\cD}'(\bR^n\setminus\{0\})$. Then $\phi\in S_k(\lambda)$ if and only if
$f(x)=r^\lambda g(x)$, where a distribution $g\in S_k(0)$, $k\in \bZ_+$.
\end{theorem}

\begin{proof}
Let us define by $\varDelta_a(\lambda)$ the operator $U_a-a^\lambda I$.
For operator of multiplication by the function $r^\lambda$ and its inverse
we will use functional symbols $r^\lambda$ and $r^{-\lambda}$, respectively.

For any distribution $f\in {\cD}'(\bR^n\setminus\{0\})$ we have
$$
(\varDelta_a(\lambda)r^\lambda f)(x)=
\big(U_a r^\lambda f-a^\lambda r^\lambda f\big)(x)
=r^\lambda(ax) f(ax)-a^\lambda r^\lambda(x) f(x)
$$
$$
=a^\lambda r^\lambda(x)(f(ax)-f(x))
=a^\lambda\big(r^\lambda(U_a-I)f\big)(x)
=(a^\lambda r^\lambda\varDelta_a(0)f)(x)\,.
$$
Thus
\begin{equation}
\label{commut}
\varDelta_a(\lambda)r^\lambda=a^\lambda r^\lambda\varDelta_a(0)\,.
\end{equation}
By successive iterations of \eqref{commut} we obtain that
\begin{equation}
\label{iter}
\varDelta_a^{k+1}(\lambda)r^\lambda
=a^{\lambda (k+1)}r^\lambda\varDelta_a^{k+1}(0)\,.
\end{equation}

If $f(x)=r(x)^\lambda g(x)$, where a distribution $g\in S_k(0)$, then \eqref{iter} implies that
$$
\varDelta_a^{k+1}(\lambda)f=\varDelta_a^{k+1}(\lambda)r^\lambda g
=\big(\varDelta_a^{k+1}(\lambda)r^\lambda\big)g
=a^{\lambda (k+1)}r^\lambda\varDelta_a^{k+1}(0)g=0,
$$
i.e., $f\in\ker\varDelta_a^{k+1}(\lambda)$ for any $a>0$, and consequently, $f\in S_k(\lambda)$.
Conversely, let $g(x)=r(x)^{-\lambda}f(x)\in {\cD}'(\bR^n\setminus\{0\})$, where a distribution
$g\in S_k(\lambda)$. According to \eqref{iter}, for any $a>0$ we have
$$
\varDelta_a^{k+1}(0)g=a^{-\lambda (k+1)}r^{-\lambda}\varDelta_a^{k+1}(\lambda)r^\lambda g
=a^{-\lambda (k+1)}r^{-\lambda}\varDelta_a^{k+1}(\lambda)f=0\,,
$$
i.e., a distribution $g\in S_k(0)$.
\end{proof}

Thus, the problem of describing the space $S_k(\lambda)$ is reduced to a particular case $\lambda=0$.

\begin{theorem}
\label{direct}
Let $f\in {\cD}'(\bR^n\setminus\{0\})$. $f\in S_k(0)$ if and only if
\begin{equation}
\label{repr-0}
f(x)=\sum_{j=0}^{k}h_j(x)\log^jr,
\end{equation}
where $h_j$, $j=0,1,\dots,k$ are homogeneous distributions.
\end{theorem}

\begin{proof}
Suppose that representation (\ref{repr-0}) holds.
For $k=0$ the statement is right, since a function $f=h_0$ is
homogeneous, and, consequently, $\varDelta_a(0)f=(U_a-I)h_0=0$.
Let us assume that for $k-1$ the statement holds. For $k$ we have
$$
f(x)=\sum_{j=0}^{k}h_j(x)\log^jr=h_k(x)\log^kr+\sum_{j=0}^{k-1}h_j(x)\log^jr,
\quad x\in \bR^n\setminus\{0\}\,,
$$
where (in view of our induction assumption) the last summand belongs to the subspace
$S_{k-1}(0)\subseteq S_k(0)$. Let us examine that $\eta=\varDelta_a(0)h_k\log^kr \in S_{k-1}(0)$
for any $a>0$. Indeed,
$$
\eta(x)=\big((U_a-I)h_k\log^kr\big)(x)=h_k(ax)\log^kr(ax)-h_k(x)\log^kr(x)=
$$
$$
=h_k(x)\big((\log r(x)+\log a)^k-\log^kr(x)\big)=
\sum_{j=0}^{k-1}\binom{k}{j}h_k(x)\log^jr(x)\log^{k-j}a\,.
$$
Thus, $\eta$ is a linear combination of functions
$h_k\log^jr$ $(j\le k-1)$ which by an induction assumption belongs to $S_{k-1}(0)$.
Therefore, $\varDelta_a^{k+1}(0)h_k\log^kr=\varDelta_a^{k}(0)\eta=0$,
i.e., $h_k\log^kr\in S_k(0)$.

Conversely, let a distribution $f\in S_k(0)$. Denote by $\omega$ the homogeneous function
$x\mapsto x/r(x)$. Here $\omega\in \bS^{n-1}$, where $\bS^{n-1}$ is a unit sphere in $\bR^n$.
Setting $\widetilde{f}(s,\omega)=\phi(\exp(s)\omega)$, one can see that a distribution
$f$ can be represented as $f(x)=\widetilde{f}(\log r,\omega)$, where
$\widetilde{f}\in {\cD}'(\bR\times \bS^{n-1})$.
Therefore,
$$
0=\varDelta_a^{k+1}(0)f=(U_a-I)^{k+1}f
=\sum_{m=0}^{k+1}\binom{k+1}{m}(-1)^{k+1-m}U_a^m f
\qquad\qquad
$$
$$
\qquad
=\sum_{m=0}^{k+1}\binom{k+1}{m}(-1)^{k+1-m}\widetilde{f}(m\log a+\log r,\omega)\,.
$$
Denote by $D$ the operator of differentiation of a function $\widetilde{f}$ with
respect of the first argument.
Next, applying the operator $(a\frac{d}{da})^{k+1}$ to the right- and left-hand side of the
last relation, setting $a=1$ and taking into account the identity
$$
\sum_{m=0}^N(-1)^{N-m}\binom{N}{m}m^N=N!\,,
$$
we obtain
$$
0=\sum_{m=0}^{k+1}\binom{k+1}{m}(-1)^{k+1-m}m^{k+1}D^{k+1}\widetilde{f}(\log r,\omega)=
(k+1)!D^{k+1}\widetilde{f}(\log r,\omega)\,.
$$
Thus, $D^{k+1}\widetilde{f}(s,\omega)=0$ for all $s\in \bR$.
Therefore, according to~\cite[\S 3.3.]{Vladimirov},
$\widetilde{f}(s,\omega)=\sum_{j=0}^{k}c_j(\omega)s^j$.
Since $\widetilde{f}\in {\cD}'(\bR\times \bS^{n-1})$, we conclude that
$$
c_j(\omega)=D^j\widetilde{f}(0,\omega)/j!\in {\cD}(\bS^{n-1})\,.
$$
Thus, $f(x)=\widetilde{f}(\log r,\omega)=\sum_{j=0}^{k}c_j(x/r(x))\log^jr$.
Denoting by $h_j$ homogeneous distribution $c_j(\omega)$,
we obtain the representation (\ref{repr-0}).

The theorem is proved.
\end{proof}

Theorems~\ref{reduction},~\ref{direct} imply the following statement.

\begin{theorem}
\label{th-1-spectr}
Let $f\in {\cD}'(\bR^n\setminus\{0\})$.
A distribution $f\in S_{k}(\lambda)$ if and only if
\begin{equation}
\label{repr-1}
f(x)=r^\lambda\sum_{j=0}^{k}h_j(x)\log^jr, \quad x\in \bR^n\setminus\{0\}\,,
\end{equation}
where $h_j$ are homogeneous distributions, $j=0,1,\dots,k$, \, $k\in \bZ_+$.
\end{theorem}

\subsection{Characterization of QAHD by a dilation operator.}\label{s2.2}

\begin{theorem}
\label{th-2-spectr}
A distribution $f\in {\cD}'(\bR^n\setminus\{0\})\cap S_{k}(\lambda)$ if and only if
$f$ is QAHD of degree $\lambda$ and of order $k$ {\rm(}i.e. satisfies Definition~{\rm\ref{de9})}.
\end{theorem}

\begin{proof}
Let $f\in S_{k}(\lambda)$. Then, due to Theorem~\ref{th-1-spectr}, a distribution $f$
is represented in the form (\ref{repr-1}). Therefore,
$$
U_af(x)=f(ax)=r(ax)^\lambda\sum_{j=0}^{k}h_j(ax)\log^jr(ax)
\qquad\qquad\qquad\qquad\qquad\qquad\qquad
$$
$$
=a^{\lambda}r(x)^\lambda\sum_{j=0}^{k}h_j(x)\big(\log r(x)+\log a\big)^j
=a^{\lambda}r^\lambda\sum_{j=0}^{k}h_j\sum_{m=0}^{j}\binom{j}{m}\log^{j-m}r\log^ma
$$
\begin{equation}
\label{repr-ashd-1}
=\sum_{m=0}^{k}\bigg(r^\lambda\sum_{j=m}^{k}\binom{j}{m}h_j\log^{j-m}r\bigg)a^{\lambda}\log^ma
=\sum_{m=0}^{k}f_{k-m}a^{\lambda}\log^ma,
\end{equation}
where in view of Theorem~\ref{th-1-spectr},
$$
f_{k-m}=r^\lambda\sum_{s=0}^{k-m}\binom{m+s}{m}h_{m+s}\log^{s}r \in S_{k-m}(\lambda),
\quad m=0,1,\dots,k;
$$
and $f_{k}=r^\lambda\sum_{s=0}^{k}h_{s}\log^{s}r=f$.
Setting in (\ref{repr-ashd-1}) successively $k=1,2,\dots$, we conclude that
$f$ is a QAHD of degree $\lambda$.

Conversely, let $f=f_k$ be a QAHD of degree $\lambda$, i.e., according to (\ref{17.1**}),
\begin{equation}
\label{17.1**-1}
U_{a}f_k(x)=f_k(ax)=a^{\lambda}f_{k}(x)+\sum_{r=1}^{k}a^{\lambda}\log^r{a} f_{k-r}(x)
\quad \text{for any} \quad a >0,
\end{equation}
where $f_{k-r}(x)$ is a QAHD of degree $\lambda$ and of order $k-r$,
\ $r=1,2,\dots,k$.

For $k=1$ we have $U_{a}f_1(x)=a^{\lambda}f_{1}(x)+a^{\lambda}\log{a} f_{0}(x)$,
where $f_{0}$ is HD. Therefore, $(U_a-a^\lambda I)f_1=a^{\lambda}\log{a} f_{0}$,
where $(U_a-a^\lambda I)f_0=0$. Thus, $(U_a-a^\lambda I)^2f_1=0$.

Suppose that $(U_a-a^\lambda I)^{j+1}f_j=0$ for $j=2,3,\dots,k-1$. Then according to
(\ref{17.1**-1}), $(U_a-a^\lambda I)f_k=\sum_{r=1}^{k}a^{\lambda}\log^r{a} f_{k-r}$
for any $a >0$. Taking into account our assumption, one can conclude that
$(U_a-a^\lambda I)^{k+1}f_k=\sum_{r=1}^{k}a^{\lambda}\log^r{a} (U_a-a^\lambda I)^{k}f_{k-r}=0$,
i.e., $f=f_k\in S_{k}(\lambda)$.
\end{proof}

Now we present another proof of the sufficiency in Theorem~\ref{th-2-spectr}.
Denote $f=(f_k)_{k=0}^\infty$, where $f_k\in {\cD}'(\bR^n)$. Let us define
the dilatation operator $U_{a}$ on $f=(f_k)_{k=0}^\infty$ as
\begin{equation}
\label{100}
\big(U_{a}f\big)_k(x)=f_{k}(ax)=f_k(ax_1,\dots,ax_n), \quad x\in \bR^n, \quad k=0,1,2,\dots.
\end{equation}
Here a vector component $f_k$ can be considered as a vector $f_k=(g_0,\dots,g_k,\dots)$,
where $g_j=\delta_{j\,k}$, and  $\delta_{j\,k}$ is the Kronecker symbol.

\begin{proposition}
\label{prop6-2}
If $f_k(x)$ is a QAHD of degree~$\lambda$ and of order $k$ then
\begin{equation}
\label{75-2}
\left(\left(U_a-a^\lambda I\right)^{k+1}f\right)_k=0, \quad \forall \, a>0,
\qquad k=0,1,2,\dots.
\end{equation}
\end{proposition}

\begin{proof}
Suppose that $f_k(x)$ is a QAHD of degree~$\lambda$ and of order $k$, $k\ge 1$.
Then by Definition~\ref{de9} we have
\begin{equation}
\label{101}
\big(U_{a}f\big)_k(x)=a^{\lambda}\Big(\sum_{r=0}^{k}\log^{k-r}{a} f_{r}(x)\Big),
\quad x\in \bR^n, \quad k=0,1,2,\dots.
\end{equation}
Thus the operator $R_a$ in right-hand side of relation (\ref{101}) has the
following matrix form
$$
R_a=a^\lambda
\left(
\begin{array}{cccccc}
1&\log{a}&\log^{2}{a}&\log^{3}{a}&\log^{4}{a}&...\\
0&1&\log{a}&\log^{2}{a}&\log^{3}{a}&...\\
0&0&1&\log{a}&\log^{2}{a}&...\\
0&0&0&1&\log{a}&...\\
0&0&0&0&1&...\\
&&&&&...
\end{array}
\right).
$$
Let $T$ be a operator defined as follows
\begin{equation}
T=
\left(
\begin{array}{cccccc}
0&1&0&0&0&...\\
0&0&1&0&0&...\\
0&0&0&1&0&...\\
0&0&0&0&1&...\\
&&&&&...
\end{array}
\right)
=\sum_{r=0}^\infty E_{r\,r+1},
\end{equation}
where $E_{k\,m}$ is a matrix such that $(E_{k\,m})_{ij}=\delta_{i\,k}\delta_{j\,m}$,
and $\delta_{i\,k}$, $\delta_{j\,m}$ are the Kronecker symbols, $k,m,i,j\in \bN_0$.
It is clear that
\begin{equation}
\label{102}
E_{k\,m}E_{p\,q}=\left\{
\begin{array}{cc}
E_{k\,q},&\,\,\text{\,\,if\quad}m=p,\\
0,     &\,\,\text{\,\,if\quad}m\not=p.
\end{array}
\right.
\end{equation}
In view of (\ref{102}), we have
$$
T^2=\Big(\sum_{r=0}^\infty E_{r\,r+1}\Big)\Big(\sum_{s=0}^\infty E_{s\,s+1}\Big)=
\sum_{r,s=0}^\infty E_{r\,r+1}E_{s\,s+1}=\sum_{r=0}^\infty E_{r\,r+2}.
$$
Similarly, using (\ref{102}), we obtain
$$
T^m=\sum_{r=0}^\infty E_{r\,r+m},\quad m\in \bN.
$$
Using the last relation, the operator $R_a$ in right-hand side of relation (\ref{101})
can be rewritten as
\begin{equation}
\label{103}
R_a=a^\lambda\sum_{s=0}^\infty\log^s{a}\,T^s=a^\lambda\left(1-(\log{a})\,T\right)^{-1}.
\end{equation}

Formulas (\ref{100}), (\ref{101}), (\ref{103}) imply that
$$
R_a-a^\lambda I=a^\lambda\sum_{s=1}^\infty \log^s{a}\,T^s
=a^\lambda\log{a}\,T\sum_{s=0}^\infty \log^s{a}\,T^s.
$$
Hence
$$
(R_a-a^\lambda I)^{k+1}=(a^\lambda\log{a})^{k+1}
T^{k+1}\bigg(\sum_{s=0}^\infty \log^s{a}\,T^s\bigg)^{k+1}.
$$

Since operators $T$ and $\sum_{s=0}^\infty \log^s{a}\,T^s$
commute, taking into account that
$$
(Tf)_k=f_{k-1},\quad (T^2f)_k=f_{k-2},\dots,(T^kf)_k=f_{0},\quad (T^{k+1}f)_k=0,
$$
and $U_af=R_af$ for $f=(f_k)_{k=0}^\infty$, where $f_k$ is a QAHD of degree~$\lambda$
and of order $k$, we get
$$
\left((U_a-a^\lambda I)^{k+1}f\right)_k
=(a^\lambda\log{a})^{k+1}\bigg(T^{k+1}\Big(\sum_{s=0}^\infty \log^s{a}\,T^s\Big)^{k+1}f\bigg)_k
$$
$$
=(a^\lambda\log{a})^{k+1}\bigg(\Big(\sum_{s=0}^\infty\log^s{a}\,T^s\Big)^{k+1}T^{k+1}f\bigg)_k=0,
\quad k=0,1,2,\dots.
$$
\end{proof}

\section{Characterization of QAHD by a differential operator}\label{s3}

The characterization of a multidimensional QAHD by differential operator
(\ref{17.1***-00}) is given by the theorem which generalizes the well-known classical 
the Euler theorem for homogeneous distributions (for example, see~\cite[Ch.III,\S 3.1.]{G-Sh}).
\begin{theorem}
\label{th6-1}
{\rm (~\cite[Theorems~5.2., Remark~5.1.]{Sh})}
$f_k(x)$ is a QAHD of degree~$\lambda$ and of order $k$, $k\ge 1$
if and only if
\begin{equation}
\label{75-1}
\bigg(\sum_{j=1}^{n}x_j\frac{\partial}{\partial x_j}-\lambda\bigg)^{k+1}f_{k}(x)=0.
\end{equation}
\end{theorem}

\begin{proof}
Let $f_{k}\in {\cD}'(\bR^n)$ be a QAHD of degree $\lambda$ and of
order~$k$. According to Definition~\ref{de9},
there are QAHD $f_{j}$ of degree $\lambda$ and of order $j$, $j=0,1,2,\dots,k-1$, and
QAHD $f_{k-s-r}^{(k-s)}$ of degree $\lambda$ and of order $k-s-j$, $s=0,1,2,\dots,k-2$,
$r=2,\dots,k-s$ such that
\begin{equation}
\label{76}
\begin{array}{rcl}
\displaystyle
U_{a}f_k(x)&=&a^{\lambda}f_{k}(x)+a^{\lambda}\log{a} f_{k-1}(x)
+\sum_{r=2}^{k}a^{\lambda}\log^r{a} f_{k-r}^{(k)}(x), \medskip \\
\displaystyle
U_{a}f_{k-1}(x)&=&a^{\lambda}f_{k-1}(x)+a^{\lambda}\log{a} f_{k-2}(x)
+\sum_{r=2}^{k-1}a^{\lambda}\log^r{a} f_{k-1-r}^{(k-1)}(x), \medskip \\
\displaystyle
U_{a}f_{k-2}(x)&=&a^{\lambda}f_{k-2}(x)+a^{\lambda}\log{a} f_{k-3}(x)
+\sum_{r=2}^{k-2}a^{\lambda}\log^r{a} f_{k-2-r}^{(k-2)}(x), \medskip \\
\hdotsfor{3} \\
\displaystyle
U_{a}f_1(x)&=&a^{\lambda}f_{1}(x)+a^{\lambda}\log{a} f_{0}(x),
\medskip \\
\displaystyle
U_{a}f_0(x)&=&a^{\lambda}f_{0}(x). \\
\end{array}
\end{equation}
Differentiating (\ref{76}) with respect to $a$ and setting $a=1$, we derive the system
\begin{equation}
\label{75}
\begin{array}{rcl}
\displaystyle
\sum_{j=1}^{n}x_j\frac{\partial f_{k}}{\partial x_j}&=&\lambda f_{k}+f_{k-1}, \\
\displaystyle
\sum_{j=1}^{n}x_j\frac{\partial f_{k-1}}{\partial x_j}&=&
\lambda f_{k-1}+f_{k-2}, \\
\hdotsfor{3} \\
\displaystyle
\sum_{j=1}^{n}x_j\frac{\partial f_{1}}{\partial x_j}&=&\lambda f_{1}+f_{0}, \\
\displaystyle
\sum_{j=1}^{n}x_j\frac{\partial f_{0}}{\partial x_j}&=&\lambda f_{0}. \\
\end{array}
\end{equation}
It is clear that system (\ref{75}) implies (\ref{75-1}).

Conversely, let $f_{k}\in {\cD}'(\bR^n)$ be a distribution satisfying
system (\ref{75-1}). Denoting
$$
\bigg(\sum_{j=1}^{n}x_j\frac{\partial}{\partial x_j}-\lambda\bigg)^{s}f_{k}=f_{k-s},
\quad s=1,2,\dots,k,
$$
we obtain $\big(\sum_{j=1}^{n}x_j\frac{\partial}{\partial x_j}-\lambda\big)f_{0}=0$,
$\big(\sum_{j=1}^{n}x_j\frac{\partial}{\partial x_j}-\lambda\big)f_{1}=f_{0}$,
\dots,
$\big(\sum_{j=1}^{n}x_j\frac{\partial}{\partial x_j}-\lambda\big)f_{k-1}=f_{k-2}$.
Thus, there are distributions $f_{j}\in {\cD}'(\bR^n)$,
$j=0,1,2,\dots,k-1$ such that system (\ref{75}) holds. Now we prove by
induction that $f_{k}$ is a QAHD of degree $\lambda$ and of order~$k$.

For $k=0$ this statement follows from Theorem~\ref{th5}.
If $k=1$ then the following system of equations
\begin{equation}
\label{76.1}
\sum_{j=1}^{n}x_j\frac{\partial f_{1}}{\partial x_j}=\lambda f_{1}+f_{0},
\quad
\sum_{j=1}^{n}x_j\frac{\partial f_{0}}{\partial x_j}=\lambda f_{0}
\end{equation}
holds. Here, in view of Theorem~\ref{th5}, the second equation in
(\ref{76.1}) implies that $f_{0}$ is a HD.

Consider the function
$$
g_1(a)=f_1(ax_1,\dots,ax_n)-a^{\lambda}f_{1}(x)-a^{\lambda}\log{a} f_{0}(x).
$$
It is clear that $g_1(1)=0$. By differentiation
we have
\begin{equation}
\label{76.2}
g_1'(a)=\sum_{j=1}^{n}x_j\frac{\partial f_1}{\partial x_j}(ax_1,\dots,ax_n)
-\lambda a^{\lambda-1}f_{1}(x)-(\lambda a^{\lambda-1}\log{a}
+a^{\lambda-1})f_{0}(x)
\end{equation}
Applying the first relation in (\ref{76.1}) to the arguments $ax_1,\dots,ax_n$
we find that
\begin{equation}
\label{76.3}
\sum_{j=1}^{n}x_j\frac{\partial f_{1}}{\partial x_j}(ax_1,\dots,ax_n)=
\frac{\lambda}{a}f_{1}(ax_1,\dots,ax_n)+\frac{1}{a}f_{0}(ax_1,\dots,ax_n). \\
\end{equation}
Substituting (\ref{76.3}) into (\ref{76.2}) and taking into account that
$\frac{1}{a}f_{0}(ax_1,\dots,ax_n)=a^{\lambda-1}f_{0}$, we find that $g_1(a)$
satisfies the differential equation with the initial data
\begin{equation}
\label{76.4}
g_1'(a)=\frac{\lambda}{a}g_1(a), \qquad g_1(1)=0.
\end{equation}
Obviously, its solution is $g_1(a)=0$. Thus
$g_1(a)=f_1(ax_1,\dots,ax_n)-a^{\lambda}f_{1}(x)
-a^{\lambda}\log{a} f_{0}(x)=0$, i.e, $f_1(x)$ is a QAHD of order $k=1$.

Let us assume that for $k-1$ the theorem holds, i.e., if $f_{k-1}$
satisfies all equations in (\ref{75}) except the first one, then
$f_{k-1}$ is a QAHD of degree $\lambda$ and of order~$k-1$.

Now, we suppose that there exist distributions $f_{k-1},\dots,f_{0}$
such that (\ref{75}) holds. Note that in view of our assumption,
$f_{k-1}$ is a QAHD of order $k-1$.

Consider the function
\begin{equation}
\label{76.5*}
g_k(a)=f_k(ax_1,\dots,ax_n)-a^{\lambda}f_{k}(x)-a^{\lambda}\log{a} f_{k-1}(x).
\end{equation}
It is clear that $g_k(1)=0$. By differentiation we have
\begin{equation}
\label{76.5}
g_k'(a)=\sum_{j=1}^{n}x_j\frac{\partial f_k}{\partial x_j}(ax_1,\dots,ax_n)
-\lambda a^{\lambda-1}f_{k}(x)
-(\lambda a^{\lambda-1}\log{a}+a^{\lambda-1})f_{k-1}(x)
\end{equation}
Applying the first relation in (\ref{75}) to the arguments $ax_1,\dots,ax_n$
we find that
\begin{equation}
\label{76.6}
\sum_{j=1}^{n}x_j\frac{\partial f_{k}}{\partial x_j}(ax_1,\dots,ax_n)=
\frac{\lambda}{a}f_{k}(ax_1,\dots,ax_n)+\frac{1}{a}f_{k-1}(ax_1,\dots,ax_n). \\
\end{equation}
Substituting (\ref{76.6}) into (\ref{76.5}) and taking into account that
according to our assumption, $f_{k-1}$ is a QAHD of order $k-1$, i.e.,
$$
U_{a}f_{k-1}(x)=f_{k-1}(ax_1,\dots,ax_n)=a^{\lambda}f_{k-1}(x)
+\sum_{r=1}^{k-1}a^{\lambda}\log^r{a} f_{k-1-r}^{(k-1)}(x),
$$
where $f_{k-1-r}^{(k-1)}(x)$ is a QAHD of order $k-1-r$, \ $r=1,2,\dots,k-1$,
we find that $g_k(a)$ satisfies the linear differential equation
\begin{equation}
\label{76.7}
g_k'(a)=\frac{\lambda}{a}g_k(a)
+\sum_{r=1}^{k-1}a^{\lambda-1}\log^r{a} f_{k-1-r}^{(k-1)}(x), \qquad g_1(1)=0.
\end{equation}
Now it is easy to see that its general solution has the form
$$
g_k(a)=\sum_{r=1}^{k-1}a^{\lambda}\log^{r+1}{a}\frac{ f_{k-1-r}^{(k-1)}(x)}{r+1}
+a^{\lambda}C(x),
$$
where $C(x)$ is a distribution. Taking into account that $g_1(1)=0$, we
calculate $C(x)=0$. Thus
\begin{equation}
\label{76.8}
g_k(a)=\sum_{r=1}^{k-1}a^{\lambda}\log^{r+1}{a}
\frac{ f_{k-1-r}^{(k-1)}(x)}{r+1}.
\end{equation}

By substituting (\ref{76.8}) into (\ref{76.5*}), we find
\begin{equation}
\label{76.9}
U_{a}f_k(x)=a^{\lambda}f_{k}(x)-a^{\lambda}\log{a} f_{k-1}(x)
+\sum_{r=2}^{k}a^{\lambda}\log^{r}{a}\frac{ f_{k-r}^{(k-1)}(x)}{r},
\end{equation}
where by our assumption $f_{k-1}$ is a QAHD of order $k-1$, and,
consequently, $f_{k-r}^{(k-1)}(x)$ is a QAHD of order $k-r$, \ $r=2,\dots,k$.
Thus, in view of Definition~\ref{de9}, $f_k$ is a QAHD of order $k$.

By the induction axiom, the theorem is proved.
\end{proof}

Theorem~\ref{th6-1} is a special case of Proposition~2.31 from Grudzinski's book~\cite{Gr}.

\section{Conclusion}\label{s4}

Theorems~\ref{th-1-spectr},~\ref{th-2-spectr},~\ref{th6-1} imply the following statement.

\begin{theorem}
\label{th6-3}
Let $f\in {\cD}'(\bR^n\setminus\{0\})$.
The following statements are equivalent

$(i)$ $f_k\in {\cD}'(\bR^n)$ is a QAHD of degree~$\lambda$ and of order $k$
{\rm(}in a sense of Definition~\ref{de9}{\rm)}, i.e.,
$$
U_{a}f_k(x)=f_{k}(ax)
=a^{\lambda}f_{k}(x)+\sum_{r=1}^{k}a^{\lambda}\log^r{a} f_{k-r}(x),
\quad x\in \bR^n,
$$
where $f_{k-r}(x)$ is a QAHD of degree $\lambda$ and of order $k-r$, \ $r=1,2,\dots,k$;

$(ii)$ $\big(U_a-a^\lambda I\big)^{k+1}f_k(x)=0$;

$(iii)$ $\big(\sum_{j=1}^{n}x_j\frac{\partial}{\partial x_j}-\lambda\big)^{k+1}f_{k}(x)=0$;

$(iv)$ $f_k(x)=r^\lambda\sum_{j=0}^{k}h_j(x)\log^jr$, where $h_j$ are homogeneous distributions,
$j=0,1,\dots,k$; \, $k\in \bZ_+$.
\end{theorem}

\end{document}